\documentclass[12pt, reqno]{amsart}
\usepackage{amsmath, amsthm, amscd, amsfonts, amssymb, graphicx, color}
\usepackage[bookmarksnumbered, colorlinks, plainpages]{hyperref}

\newtheorem{theorem}{Theorem}[section]
\newtheorem{lemma}[theorem]{Lemma}

\newtheorem{corollary}[theorem]{Corollary}
\theoremstyle{definition}

\newtheorem{example}[theorem]{Example}

\theoremstyle{remark}

\numberwithin{equation}{section}

\begin{document}
\setcounter{page}{1}

%-------------------------- Pleased do not change the following line-------------------------------------------
%\noindent \textcolor[rgb]{0.99,0.00,0.00}{This is a submission to one of journals of TMRG: BJMA or AFA}\\[.5in]
%--------------------------------------------------------------------------------------------------------------

\title[Commuting Jordan derivations on triangular rings are zero]{Commuting Jordan derivations on triangular rings are zero}

\author[Amin Hosseini]{Amin Hosseini$^*$}

\address{ Amin Hosseini, Kashmar Higher Education Institute, Kashmar, Iran}
\email{\textcolor[rgb]{0.00,0.00,0.84}{a.hosseini@kashmar.ac.ir}}

\author[Wu Jing]{Wu Jing}

\address{ Wu Jing, Department of Mathematics and Computer Science, Fayetteville State University, Fayetteville NC 28301 USA}
\email{\textcolor[rgb]{0.00,0.00,0.84}{wjing@uncfsu.edu}}

%\dedicatory{This paper is dedicated to Professor ...}



\date{Received: xxxxxx; Revised: yyyyyy; Accepted: zzzzzz.
\newline \indent $^{*}$ Corresponding author}

\begin{abstract}
The main purpose of this article is to show that every commuting Jordan derivation on triangular rings (unital or not) is identically zero. Using this result, we prove that if $\mathcal{A}$ is a 2-torsion free ring such that it is either
semiprime or satisfies Condition (P), then every commuting Jordan derivation from $\mathcal{A}$ into itself, under certain conditions, is identically zero.
\end{abstract} \maketitle

\section{Introduction and Preliminaries}
Let $\mathcal{R}$ be an associative ring with center $Z(\mathcal{R})$. We write $[x, y] = xy - yx$ for all $x, y \in \mathcal{R}$. A ring $\mathcal {R}$ is said to be 2-torsion free, if whenever $2a=0$, with $a\in \mathcal {R}$, then $a=0$. A ring $\mathcal {R}$ is called a prime ring if $a\mathcal {R}b=\{0\}$ implies $a=0$ or $b=0$, and also a ring $\mathcal {R}$ is called a semiprime ring if $a\mathcal {R}a=\{0\}$ implies $a=0$. A ring $\mathcal{R}$ satisfies Condition (P) if $x a x = 0$ for all $x \in \mathcal{R}$, then $a = 0$. Clearly, every unital ring satisfies this condition. For non-unital rings that  satisfy Condition (P), the reader is referred to \cite{F}.

Recall that an additive map $\Delta :\mathcal {R}\to \mathcal {R}$ is called a derivation if $\Delta (ab)=\Delta (a)b+a\Delta (b)$ for all $a, b\in \mathcal {R}$. $\Delta $ is said to be a Jordan derivation if $\Delta (a^2)=\Delta (a)a+a\Delta (a)$ for all $a\in \mathcal {R}$. Also, $\Delta $ is called a left (resp. right) Jordan derivation if $\Delta(a^2)=2a\Delta (a)$ (resp. $\Delta (a^2)=2\Delta (a)a$) for any $a\in \mathcal {A}$. For more details about left (Jordan) derivations, see, e.g. \cite{B-V, V9}.

The first result on a Jordan derivation to be a derivation is due to  Herstein \cite{He} who proved that every
Jordan derivation on a 2-torsion free prime ring is a derivation. %A brief proof of Herstein's result can be found in [3].
Cusack \cite{Cu} generalized Herstein's result to 2-torsion free semiprime rings (see also \cite{B1} for an alternative proof). In 2008, Vukman \cite{V9} studied left Jordan derivations on semiprime rings. In that article he showed that if $\mathcal{R}$ is a 2-torsion free semiprime ring and $\Delta : \mathcal{R} \rightarrow \mathcal{R}$ is a left Jordan derivation, then $\Delta$ is a derivation which maps $\mathcal{R}$ into $Z(\mathcal{R})$. In recent years, the characterizations of Jordan derivations on triangular rings have been studied. For example, it was proved in \cite{F} that, under some conditions, every Jordan derivation on a triangular ring (without assuming unity) is a derivation. For more studies concerning Jordan derivations we refer the reader to \cite{G, Z} and the references therein.

A mapping $\mathfrak{F}:\mathcal{R} \rightarrow \mathcal{R}$ is said to be centralizing on a subset $\mathfrak{X}$ of $\mathcal{R}$ if $[\mathfrak{F} (x), x] \in Z(\mathcal{R})$ for all $x \in \mathfrak{X}$. In particular, if $[\mathfrak{F} (x), x] = 0$ for all $x \in \mathfrak{X}$, then $\mathfrak{F}$ is called commuting on $\mathfrak{X}$. In the last few decades, the  commuting maps has been one of the most active topics in the study of mappings on rings. For commuting maps, we refer the readers to the very nice survey paper \cite{B}. The history of commuting and centralizing mappings goes back to 1955 when Divinsky \cite{D} proved that a simple artinian ring is commutative if it has a commuting nontrivial automorphism.  Two years later, Posner \cite{P} achieved the first result on commuting derivations which claims that if $\Delta$ is a commuting derivation on a prime ring $\mathcal{R}$, then either $\mathcal{R}$ is commutative or $\Delta$ is zero.

In this paper, we aim to study commuting Jordan derivations on triangular rings. Different from Ponser's result, we show that every commuting Jordan derivation on triangular rings has to be zero. We also want to mention that our result of this paper does not require the triangular rings to be unital.

Throughout this paper, $\mathcal {A}$ and $\mathcal {B}$ are associative rings and $\mathcal {M}$ is an $(\mathcal {A}, \mathcal {B})$-bimodule.
Recall that a triangular ring $\mathfrak{T} = Tri(\mathcal{A}, \mathcal{M}, \mathcal{B})$ is a ring of the form
 $$\mathfrak{T} = Tri(\mathcal{A}, \mathcal{M}, \mathcal{B})=\left \{ \left [\begin{array}{cc}
 a & m\\
 0 & b
 \end{array}\right ]: a\in \mathcal {A}, m\in \mathcal {M}, b\in \mathcal {B}\right \}$$
 under the usual matrix addition and multiplication (see \cite{C}). Note that $\mathfrak{T} = Tri(\mathcal{A}, \mathcal{M}, \mathcal{B})$ is unital if and only if both $\mathcal{A}$ and $\mathcal{B}$ are unital. We remark
that there exist many triangular rings without unity. For example, if $\mathcal{A}$ is a ring
without unity then every upper triangular matrix ring over $\mathcal{A}$ does not contain unity.
\\
We set \begin{displaymath} \mathfrak
 {T}_{11}=\left \{ \left[ \begin{array}{cc}
 a & 0\\
 0 & 0
 \end{array}\right]: a\in \mathcal {A} \right \},
 \end{displaymath}
 \begin{displaymath}\ \ \ \ \mathfrak{T}_{12}=\left \{
 \left[ \begin{array}{cc}
 0 & m\\
 0 & 0
 \end{array}\right]: m\in \mathcal {M} \right \},
 \end{displaymath}
 and
 \begin{displaymath} \mathfrak{T}_{22}=\left \{
 \left[ \begin{array}{cc}
 0 & 0\\
 0 & b
 \end{array}\right]: b\in \mathcal {B} \right \}.
 \end{displaymath}
 Then we may write $\mathfrak{T}=\mathfrak{T}_{11}\oplus \mathfrak
 {T}_{12}\oplus \mathfrak{T}_{22}$ and every element $A\in\mathfrak
 {T}$ can be written as $A=A_{11}+A_{12}+A_{22}$, where $A_{ij}\in \mathfrak{T}_{ij}$, $i,j\in \{1,2\}$.

Let $\mathcal{A}$ and $\mathcal{B}$ be algebras. A left (resp. right) $\mathcal{A}$-module $\mathcal{M}$ is said to be left (resp. right) faithful if
$a=0$ is the only element in $\mathcal{A}$ satisfying $a \mathcal{M}=0$ (resp. $\mathcal{M} a=0$). An $(\mathcal{A}, \mathcal{B})$-bimodule
$\mathcal{M}$ is said to be faithful if $\mathcal{M}$ is both a faithful left $\mathcal{A}$-module and a faithful right $\mathcal{B}$-module
(see \cite{C} for more details).  A module $\mathcal{M}$ is said to be $n$-torsion free, where $n > 1$ is an integer, if for $x \in \mathcal{M}$, $n x = 0$ implies that $x = 0$.

Now, we state the main result of this paper. Let $\mathcal{A}$ and $\mathcal{B}$ be 2-torsion free rings such that each of them is either
semiprime or satisfies Condition (P) and let $\mathcal{M}$ be a 2-torsion free faithful $(\mathcal{A}, \mathcal{B})$-bimodule such that if $\mathcal {A}m=\{ 0\}$ (resp. $m\mathcal{B}=\{0\}$) for some $m\in \mathcal {M}$, then $m=0$. If $\Delta$ is a commuting Jordan derivation on the triangular ring $\mathfrak{T} = Tri(\mathcal{A}, \mathcal{M}, \mathcal{B})$, then $\Delta$ is zero.

As stated above, Posner proved in \cite[Lemma 3]{P} that if $\mathcal{R}$ is a prime ring and $d$ is a commuting derivation of $\mathcal{R}$, then $\mathcal{R}$ is commutative or $d$ is zero. As a corollary of the main theorem of this paper, we show that every commuting Jordan derivation on a 2-torsion free ring which either is semiprime or satisfies Condition (P)is identically zero under certain conditions.
%As a corollary of this theorem, we show that if $\mathcal{A}$ is a 2-torsion free ring such that it is either semiprime or satisfies Condition (P), then every commuting Jordan derivation from $\mathcal{A}$ into itself, under certain conditions, is identically zero.
Some other related results are also presented.

\section{Results and Proofs}

We begin our discussion with the following useful lemmas which we will use frequently to prove the main result of this paper:

\begin{lemma}\label{1}%\cite[Proposition 1.4]{Z}
Let $\mathcal{R}$ be a 2-torsion free ring and let $\Delta:\mathcal{R} \rightarrow \mathcal{R}$ be a commuting Jordan derivation, that is $\Delta(x^2) = \Delta(x)x + x \Delta(x)$ and $\Delta(x)x=x\Delta (x)$ for all $x \in \mathcal{R}$.  Then for any $x, y \in \mathcal{R}$, the following hold:
\begin{enumerate}
	\item   [(i)] $\Delta(xy + yx) = 2\big(x\Delta(y) + y \Delta(x)\big)$;
	\item   [(ii)] $\Delta(xy + yx) = 2\big(\Delta(x)y + \Delta(y)x\big);$
	\item   [(iii)] $\Delta(xyx) = 3 \Delta(x)yx + \Delta(y)x^2 - \Delta(x)xy$;
	\item   [(iv)] $\Delta(xyx) = x^2 \Delta(y) + 3 xy \Delta(x) - yx \Delta(x)$;
	%\item   [(v)] $\Delta(xyx) = \Delta(x)yx + x \Delta(y)x + xy\Delta(x)$;
	%\item   [(vi)] $\Delta(xyz + zyx) = 3\Delta(x)yz + 3 \Delta(z) yx + \Delta(y)(xz + xz) - \Delta(x)zy - \Delta(z)xy$;
	%\item   [(vii)] $\Delta(xyz + zyx) = (xz + zx)\Delta(y) + 3 xy \Delta(z) + 3zy\Delta(x) - yx\Delta(z) - yz \Delta(x)$.
\end{enumerate}
\end{lemma}
\begin{proof}
It is clear that $\Delta$ is both a left Jordan derivation and a right Jordan derivation. Equalities (i) and (iv) have been proved in \cite[Proposition 1.1]{B-V}. Similarly, Equalities (ii) and (iii) can be obtained for the right Jordan derivations, and in order to make
this paper self-contained, we prove them here. The Equality (ii) follows immediately from $\Delta(x^2) = 2\Delta(x)x$ by the linearization (i.e., substituting $x+y$ for $x$). Let us prove (iii). From (ii), we have
\begin{align*}
\Delta(x(xy + yx) + (xy + yx)x) & =  2 \Delta(x)(xy + yx) + 2 \Delta(xy + yx)x \\ & = 6 \Delta(x) yx + 2 \Delta(x) xy + 4 \Delta(y) x^2,
\end{align*}
for all $x, y \in \mathcal{R}$. On the other hand we have
\begin{align*}
\Delta(x(xy + yx) + (xy + yx)x) & = \Delta(x^2 y + y x^2) + 2 \Delta(xyx) \\ & = 2 \Delta(x^2)y + 2\Delta(y)x^2 + 2 \Delta(xyx) \\ & = 4 \Delta(x)xy + 2\Delta(y)x^2 + 2 \Delta(xyx).
\end{align*}
In comparison and using the assumption that $\mathcal{R}$ is a 2-torsion free ring, we get
\begin{align*}
\Delta(xyx) = 3 \Delta(x)yx  + \Delta(y)x^2 - \Delta(x)xy,
\end{align*}
for all $x, y \in \mathcal{R}$, as required.
\end{proof}

\begin{lemma}\label{1*} Let $\mathcal{R}$ be a 2-torsion free ring and let $x^2 a= 0$ for all $x \in \mathcal{R}$ and some $a \in Z(\mathcal{R})$. The following expressions hold:
\begin{enumerate}
	\item   [(i)] If $\mathcal{R}$ is semiprime, then $a = 0$.
\item [(ii)] If $\mathcal{R}$ satisfies Condition (P), then $a = 0$.\end{enumerate}
\end{lemma}
\begin{proof} (i) Replacing $x$ by $x + y$ in the equation $x^2 a = 0$ and then using this equation, we get that $(xy + yx)a = 0$ for all $x, y \in \mathcal{R}$. Putting $y = a$ in the previous equation and using the assumption that $a \in Z(\mathcal{R})$, we obtain that $2 a x a = 0$ for all $x \in \mathcal{R}$. Since $\mathcal{R}$ is 2-torsion free, $a x a = 0$ for all $x \in \mathcal{R}$ and semiprimeness of $\mathcal{R}$ yields that $a= 0$, as required.\\
(ii) Since $x^2 a= 0$ for all $x \in \mathcal{R}$ and some $a \in Z(\mathcal{R})$, we deduce that $x a x = 0$ for all $x \in \mathcal{R}$. Condition (P) of $\mathcal{R}$ implies that $a =0$.
\end{proof}

\begin{lemma} \label {2*} \cite[Lemma 3]{V8} Let $\mathcal{R}$ be a semiprime ring and let $f : \mathcal{R} \rightarrow  \mathcal{R}$ be an additive mapping. If either $f (x)x = 0$ or $x f (x) = 0$ holds for all $x \in \mathcal{R}$, then $f = 0$.
\end{lemma}

We are now in a position to prove our main result.

\begin{theorem} \label{2} Let $\mathcal{A}$ and $\mathcal{B}$ be 2-torsion free rings such that each of them is either
semiprime or satisfies Condition (P) and let $\mathcal{M}$ be a 2-torsion free faithful $(\mathcal{A}, \mathcal{B})$-bimodule such that if $m\in \mathcal {M}$ and  $\mathcal {A}m=\{ 0\}$ (resp. $m\mathcal{B}=\{0\}$), then $m=0$.
If $\Delta$ is a commuting Jordan derivation on the triangular ring $\mathfrak{T} = Tri(\mathcal{A}, \mathcal{M}, \mathcal{B})$,
then $\Delta$ is zero.
\end{theorem}
\begin{proof} Without loss of generality, we assume that $\mathcal{A}$ is  a semiprime ring and $\mathcal{B}$ satisfies Condition (P). For the cases that  both  rings $\mathcal{A}$ and $\mathcal{B}$ are semiprime or both satisfy Condition (P) or ring $\mathcal {B}$ is semiprime and ring $\mathcal {A}$ satisfies Condition (P), we leave it to the interested reader. The proof is divided into the following six steps. \\

\textbf{Step 1}.  For every $A_{11}=\left [\begin{array}{cc}
a & 0\\
0 & 0
\end{array}\right ] \in \mathfrak{T}_{11}$ and $A_{22}=\left [\begin{array}{cc}
0 & 0\\
0 & b
\end{array}\right ]\in \mathfrak{T}_{22}$, we have
\begin{enumerate}
\item [(i)] $[\Delta(A_{11})]_{22} = 0$;
\item [(ii)] $[\Delta(A_{11})]_{12} = 0$;
\item [(iii)] $[\Delta(B_{22})]_{11} = 0$. \\
\end{enumerate}
 Notice that $A_{11}B_{22} = B_{22}A_{11} = 0$. It follows from Lemma \ref{1} (i) that
\begin{align*}
0 & = \Delta(A_{11}B_{22} + B_{22}A_{11}) = 2 \big(\Delta(A_{11})B_{22} + \Delta(B_{22})A_{11}\big) \\ & = 2 \left [\begin{array}{cc}
	[\Delta(A_{11})]_{11} & [\Delta(A_{11})]_{12}\\
	0 & [\Delta(A_{11})]_{22}
	\end{array}\right ]\left [\begin{array}{cc}
	0 & 0\\
	0 & b
	\end{array}\right ] + 2 \left [\begin{array}{cc}
	[\Delta(B_{22})]_{11} & [\Delta(B_{22})]_{12}\\
	0 & [\Delta(B_{22})]_{22}
	\end{array}\right ] \left [\begin{array}{cc}
	a & 0\\
	0 & 0
	\end{array}\right ]\\ & = \left [\begin{array}{cc}
	0 & 2[\Delta(A_{11})]_{12} b\\
	0 & 2[\Delta(A_{11})]_{22} b
	\end{array}\right ] + \left [\begin{array}{cc}
	2[\Delta(B_{22})]_{11} a & 0\\
	0 & 0
	\end{array}\right ] \\ & = \left [\begin{array}{cc}
	2[\Delta(B_{22})]_{11} a & 2[\Delta(A_{11})]_{12} b\\
	0 & 2[\Delta(A_{11})]_{22} b
	\end{array}\right ],
\end{align*}
which implies that
$$\left [\begin{array}{cc}
	2[\Delta(B_{22})]_{11} a & 2[\Delta(A_{11})]_{12} b\\
	0 & 2[\Delta(A_{11})]_{22} b
	\end{array}\right ] = 0.$$

It follows that $[\Delta(B_{22})]_{11} a=0$ for any $a\in \mathcal {A}$,  $[\Delta(A_{11})]_{12} b=0$ and $[\Delta(A_{11})]_{22} b=0$ for any $b\in \mathcal {B}$. Since ring $\mathcal {A}$ is semiprime and ring $\mathcal {B}$ satisfies Condition (P), we have $[\Delta(B_{22})]_{11} = 0$ and $[\Delta(A_{11})]_{22} = 0$. Using the fact that $m=0$ is the only element in $\mathcal {M}$ such that $m\mathcal {B}=\{ 0\}$, we see that $[\Delta(A_{11})]_{12} = 0$.  \\

\textbf{Step 2}. For every $A_{11}=\left [\begin{array}{cc}
a & 0\\
0 & 0
\end{array}\right ] \in \mathfrak{T}_{11}$ and $A_{12}=\left [\begin{array}{cc}
0 & m \\
0 & 0
\end{array}\right ] \in \mathfrak{T}_{12}$, we have
\begin{enumerate}
\item [(i)] $[\Delta(A_{12})]_{11} = 0$;
\item [(ii)] $[\Delta(A_{11})]_{11} = 0$. \\
\end{enumerate}

In order to show $[\Delta(A_{12})]_{11} = 0$, we first show that  $[\Delta(A_{12})]_{11} \in Z(\mathcal{A})$. According to Step 1, $[\Delta(A_{11})]_{22} = 0$ which implies that $A_{12}\Delta(A_{11}) = 0$. Hence, we have
\begin{align*}
\Delta(A_{11}A_{12}) & = \Delta(A_{11}A_{12} + A_{12}A_{11}) \\ & = 2 \big(A_{11}\Delta(A_{12}) + A_{12}\Delta(A_{11})\big) \\ & = 2 A_{11}\Delta(A_{12}) \\ & = 2 \left [\begin{array}{cc}
	 a & 0\\
	0 & 0
	\end{array}\right ]\left [\begin{array}{cc}
	[\Delta(A_{12})]_{11}& [\Delta(A_{12})]_{12} \\
	0 & [\Delta(A_{12})]_{22}
	\end{array}\right ] \\ & =  \left [\begin{array}{cc}
	2a[\Delta(A_{12})]_{11}& 2a[\Delta(A_{12})]_{12} \\
	0 & 0
	\end{array}\right ],
\end{align*}
which means that
\begin{align}
\Delta(A_{11}A_{12}) = \left [\begin{array}{cc}
	2a[\Delta(A_{12})]_{11}& 2a[\Delta(A_{12})]_{12} \\
	0 & 0
	\end{array}\right ]
\end{align}
On the other hand, by  Step 1, we have
\begin{align*}
\Delta(A_{11}A_{12}) & = \Delta(A_{11}A_{12} + A_{12}A_{11}) \\ & = 2 \big(\Delta(A_{11})A_{12} + \Delta(A_{12})A_{11}\big) \\ & =
\left [\begin{array}{cc}
	2[\Delta(A_{11})]_{11}& 0 \\
	0 & 0
	\end{array}\right ]\left [\begin{array}{cc}
	0 & m \\
	0 & 0
	\end{array}\right ] + \left [\begin{array}{cc}
	2[\Delta(A_{12})]_{11}& 2[\Delta(A_{12})]_{12} \\
	0 & 2[\Delta(A_{12})]_{22}
	\end{array}\right ]\left [\begin{array}{cc}
	a & 0 \\
	0 & 0
	\end{array}\right ] \\ & = \left [\begin{array}{cc}
	2[\Delta(A_{12})]_{11}a & 2[\Delta(A_{11})]_{11}m \\
	0 & 0
	\end{array}\right ],
\end{align*}
which means that
\begin{align}
\Delta(A_{11}A_{12}) =  \left [\begin{array}{cc}
	2[\Delta(A_{12})]_{11}a & 2[\Delta(A_{11})]_{11}m \\
	0 & 0
	\end{array}\right ].
\end{align}

Comparing equalities (2.1) and (2.2) and using the assumption that $\mathcal{A}$ is a 2-torsion free ring, we get that $[\Delta(A_{12})]_{11} \in Z(\mathcal{A})$.

 Now we prove that $[\Delta(A_{12})]_{11} = 0 = [\Delta(A_{11})]_{11}$. Using Lemma \ref{1} (iii), we have
\begin{align*}
0 & = \Delta(A_{11}A_{12}A_{11}) = 3 \Delta(A_{11})A_{12}A_{11} + \Delta(A_{12})A_{11}^2 - \Delta(A_{11})A_{11}A_{12} \\ & = \left [\begin{array}{cc}
	[\Delta(A_{12})]_{11} & [\Delta(A_{12})]_{12} \\
	0 & [\Delta(A_{12})]_{22}
	\end{array}\right ]\left [\begin{array}{cc}
	a^2 & 0 \\
	0 & 0
	\end{array}\right ] - \left [\begin{array}{cc}
	[\Delta(A_{11})]_{11} & 0 \\
	0 & 0
	\end{array}\right ]\left [\begin{array}{cc}
	0 & am \\
	0 & 0
	\end{array}\right ] \\ & = \left [\begin{array}{cc}
	[\Delta(A_{12})]_{11}a^2 & -[\Delta(A_{11})]_{11} am \\
	0 & 0
	\end{array}\right ],
\end{align*}
which means that
\begin{align*}
\left [\begin{array}{cc}
	[\Delta(A_{12})]_{11}a^2 & -[\Delta(A_{11})]_{11} am \\
	0 & 0
	\end{array}\right ] = 0.
\end{align*}
Therefore, we have
\begin{align}
& [\Delta(A_{12})]_{11} a^2 = 0, \\ & [\Delta(A_{11})]_{11} am = 0,
\end{align}
for all $a \in \mathcal{A}$ and $m \in \mathcal{M}$. Since $[\Delta(A_{12})]_{11} \in Z(\mathcal{A})$,  $[\Delta(A_{12})]_{11} a^2 = 0$  for all $a \in \mathcal{A}$ implies that $[\Delta(A_{12})]_{11} =0$ by  Lemma \ref{1*} (i).

%we get that $a [\Delta(A_{12})]_{11} a = 0$ for all $a \in \mathcal{A}$ and since $\mathcal{A}$ has (P)-condition, we deduce that $[\Delta(A_{12})]_{11} = 0$. \\

It follows from equality (2.4) that  $ [\Delta(A_{11})]_{11}a = 0$ for each $a\in \mathcal {A}$  since $\mathcal{M}$ is a faithful left $\mathcal{A}$-module.   We now define a mapping $F_{11}:\mathcal{A} \rightarrow \mathcal{A}$ by  $F_{11}(a) = [\Delta(A_{11})]_{11}$, where $A_{11} = \left [\begin{array}{cc}
	a & 0 \\
	0 & 0
	\end{array}\right ]$. Obviously, mapping $F_{11}$
is additive.  Thus, we have $F_{11}(a)a = 0$ for all $a \in \mathcal{A}$. It follows from Lemma \ref{2*} that $F_{11}(a) = [\Delta(A_{11})]_{11} = 0$ for all $A_{11} \in \mathfrak{T}_{11}$, as desired.\\

\textbf{Step 3}. $[\Delta(A_{22})]_{12} = 0$ for all $A_{22} \in \mathfrak{T}_{22}$. \\

  Obviously, $A_{11}A_{22} = 0 = A_{22}A_{11}$ for all $A_{11} \in \mathfrak{T}_{11}$ and $A_{22} \in \mathfrak{T}_{22}$. Applying Lemma \ref{1} (i) and    Step 1  (i) and (iii), we have
\begin{align*}
0 & = \Delta(A_{11}A_{22} + A_{22}A_{11}) \\ & = 2 \big(A_{11}\Delta(A_{22}) + A_{22}\Delta(A_{11})\big) \\ & = 2 \left [\begin{array}{cc}
	a & 0 \\
	0 & 0
	\end{array}\right ]\left [\begin{array}{cc}
	0 & [\Delta(A_{22})]_{12} \\
	0 & [\Delta(A_{22})]_{22}
	\end{array}\right ] \\ & = \left [\begin{array}{cc}
	0 & 2 a [\Delta(A_{22})]_{12} \\
	0 & 0
	\end{array}\right ],
\end{align*}
which yields that $2 a [\Delta(A_{22})]_{12} = 0$ for all $a \in \mathcal{A}$. Since $\mathcal{M}$ is a 2-torsion free ring, $ a [\Delta(A_{22})]_{12} = 0$ for all $a \in \mathcal{A}$ and so $[\Delta(A_{22})]_{12} = 0$ for all $A_{22} \in \mathfrak{T}_{22}$, as desired.\\

\textbf{Step 4}. For every $A_{12}=\left [\begin{array}{cc}
0 & m \\
0 & 0
\end{array}\right ] \in \mathfrak{T}_{12}$ and $A_{22}=\left [\begin{array}{cc}
0 & 0 \\
0 & b
\end{array}\right ] \in \mathfrak{T}_{22}$, we have
\begin{enumerate}
\item [(i)] $[\Delta(A_{12})]_{22} = 0$;
\item [(ii)] $[\Delta(A_{22})]_{22} = 0$.\\
\end{enumerate}

We first show that $[\Delta(A_{12})]_{22} \in Z(\mathcal{B})$. On one hand, using Step 1 (iii),  Step 2 (i) and Step 3, we have the following expressions.
\begin{align*}
\Delta(A_{12}A_{22}) & = \Delta(A_{22}A_{12} + A_{12}A_{22}) \\ & = 2 \big(\Delta(A_{22})A_{12} + \Delta(A_{12})A_{22}) \\ & = 2 \left [\begin{array}{cc}
	0 & 0 \\
	0 & [\Delta(A_{22})]_{22}
	\end{array}\right ]\left [\begin{array}{cc}
	0 & m \\
	0 & 0
	\end{array}\right ] +  2 \left [\begin{array}{cc}
	0 & [\Delta(A_{12})]_{12} \\
	0 & [\Delta(A_{12})]_{22}
	\end{array}\right ]\left [\begin{array}{cc}
	0 & 0 \\
	0 & b
	\end{array}\right ] \\ & = \left [\begin{array}{cc}
	0 & 2[\Delta(A_{12})]_{12}b \\
	0 & 2[\Delta(A_{12})]_{22}b
	\end{array}\right ],
\end{align*}
that is
\begin{align}
\Delta(A_{12}A_{22}) = \left [\begin{array}{cc}
	0 & 2[\Delta(A_{12})]_{12}b \\
	0 & 2[\Delta(A_{12})]_{22}b
	\end{array}\right ]
\end{align}
for all $b \in \mathcal{B}$. On the other hand, using Lemma \ref{1} (i),  Step 1 (iii), Step 2 (i) and Step 3, we get that
\begin{align*}
\Delta(A_{12}A_{22}) & = \Delta(A_{22}A_{12} + A_{12}A_{22}) \\ & = 2 \big(A_{22}\Delta(A_{12}) + A_{12}\Delta(A_{22})\big) \\ & = 2 \left [\begin{array}{cc}
	0 & 0 \\
	0 & b
	\end{array}\right ]\left [\begin{array}{cc}
	0 & [\Delta(A_{12})]_{12} \\
	0 & [\Delta(A_{12})]_{22}
	\end{array}\right ] +  2 \left [\begin{array}{cc}
	0 & m \\
	0 & 0
	\end{array}\right ]\left [\begin{array}{cc}
	0 & 0 \\
	0 & [\Delta(A_{22})]_{22}
	\end{array}\right ] \\ & = \left [\begin{array}{cc}
	0 & 2m[\Delta(A_{22})]_{22}\\
	0 & 2b[\Delta(A_{12})]_{22}
	\end{array}\right ],
\end{align*}
i.e.,
\begin{align}
\Delta(A_{12}A_{22}) = \Big[\begin{array}{cc}
	0 & 2m[\Delta(A_{22})]_{22}\\
	0 & 2b[\Delta(A_{12})]_{22}
	\end{array}\Big]
\end{align}
for all $b \in \mathcal{B}$. Comparing equalities (2.5) and (2.6), we obtain that $[\Delta(A_{12})]_{22} \in Z(\mathcal{B})$ for all $A_{12} \in \mathfrak{T}_{12}$. Using Lemma \ref{1} (iv), Step 1 (iii), Step 2 (i) and Step 3, we have
\begin{align*}
0 & = \Delta(A_{22}A_{12}A_{22}) = A_{22}^{2} \Delta(A_{12}) + 3 A_{22}A_{12} \Delta(A_{22}) - A_{12}A_{22}\Delta(A_{22}) \\ & = \left[\begin{array}{cc}
	0 & 0\\
	0 & b^2
	\end{array}\right]\left[\begin{array}{cc}
	0 & [\Delta(A_{12})]_{12}\\
	0 & [\Delta(A_{12})]_{22}
	\end{array}\right] - \left[\begin{array}{cc}
	0 & m\\
	0 & 0
	\end{array}\right]\left[\begin{array}{cc}
	0 & 0\\
	0 & b
	\end{array}\right]\left[\begin{array}{cc}
	0 & 0\\
	0 & [\Delta(A_{22})]_{22}
	\end{array}\right] \\ & = \left[\begin{array}{cc}
	0 & -mb[\Delta(A_{22})]_{22}\\
	0 & b^2[\Delta(A_{12})]_{22}
	\end{array}\right],
\end{align*}
which yields that
\begin{align}
& b^2[\Delta(A_{12})]_{22} = 0, \ \ \ \ \ \ \ \ \ \ \ \  (b \in \mathcal{B}) \\ & mb[\Delta(A_{22})]_{22} = 0, \ \ \ \ \ \ \ \ \ \ (b \in \mathcal{B}, \ m \in \mathcal{M}).
\end{align}

Note that  $[\Delta(A_{12})]_{22} \in Z(\mathcal{B})$ for all $A_{12} \in \mathfrak{T}_{12}$. This fact along with equality (2.7) imply that $b[\Delta(A_{12})]_{22} b= 0$ for all $b \in \mathcal{B}$. Since $\mathcal{B}$ satisfies Condition (P), we get that $[\Delta(A_{12})]_{22}= 0$ for all $A_{12} \in \mathfrak{T}_{12}$. Also, it follows from equality (2.8) that $\mathcal{M} b[\Delta(A_{22})]_{22} = \{0\}$ for all $A_{22} = \left[\begin{array}{cc}
	0 & 0\\
	0 & b
	\end{array}\right] \in \mathfrak{T}_{22}$. Since $\mathcal{M}$ is a faithful right $\mathcal{B}$-module, we get that
\begin{align}
b[\Delta(A_{22})]_{22} = 0
\end{align}
for all $A_{22} = \left[\begin{array}{cc}
	0 & 0\\
	0 & b
	\end{array}\right] \in \mathfrak{T}_{22}.$
It is clear that  mapping $F_{22}: \mathcal{B} \rightarrow \mathcal{B}$ defined by $F_{22}(b) = [\Delta(A_{22})]_{22}$ is additive, where
$A_{22} = \left[\begin{array}{cc}
	0 & 0\\
	0 & b
	\end{array}\right]$. It follows from equality (2.9) that

\begin{align}
b F_{22}(b) = 0, \ \ \ \ \ \ \ \ \ (b \in \mathcal{B}).
\end{align}
Since $\Delta$ is a commuting map, one can verifies that $F_{22}(b)b = b F_{22}(b)$ for any $b \in \mathcal{B}$. Therefore,
\begin{align}
F_{22}(b)b = b F_{22}(b)=0  \ \ \ \ \ \ \ \ \ \ \ (b \in \mathcal{B}).
\end{align}

Replacing $b$ by $b_1 + b_2$ in equality (2.10) and then using equality (2.10), we have

\begin{align*}
b_1 F_{22}(b_2) + b_2 F_{22}(b_1) = 0 \ \ \ \ \ \ \ \ \ \ (b_1, b_2 \in \mathcal{B}).
\end{align*}

Multiplying the previous equality from right  by $b_1$ and then using identity  (2.11), we get that
\begin{align}
b_1 F_{22}(b_2)b_1 =0, \ \ \ \ \ \ \ \ \ \ \ \ \ \ \ \ \ (b_1, b_2 \in \mathcal{B}).
\end{align}
Equation (2.12) and  the assumption that  ring $\mathcal{B}$ satisfies Condition (P) imply that $F_{22}(b) = 0$ for all $b \in \mathcal{B}$. This means that $[\Delta(A_{22})]_{22} = 0$ for all $A_{22} \in \mathfrak{T}_{22}$. \\

\textbf{Step 5}. $[\Delta(A_{12})]_{12} = 0$ for every $A_{12} \in \mathfrak{T}_{12}$. \\

  On one hand, using Lemma \ref{1} (i), Step 1 (i), Step 2 (i) and Step 4 (i), for any $A_{11}  = \left[\begin{array}{cc}
  a & 0\\
  0 & 0
  \end{array}\right] \in \mathfrak{T}_{11}$ and any $A_{12} \in \mathfrak{T}_{12}$,  we have
\begin{align*}
\Delta(A_{11} A_{12}) & = \Delta(A_{11} A_{12} + A_{12} A_{11}) \\ & = 2 \big(A_{11} \Delta(A_{12}) + A_{12} \Delta(A_{11})\big) \\ & = \left[\begin{array}{cc}
	2a & 0\\
	0 & 0
	\end{array}\right]\left[\begin{array}{cc}
	0 & [\Delta(A_{12})]_{12}\\
	0 & 0
	\end{array}\right] \\ & =   \left[\begin{array}{cc}
	0 & 2a[\Delta(A_{12})]_{12}\\
	0 & 0
	\end{array}\right],
\end{align*}
which means that
\begin{align}
\Delta(A_{11} A_{12}) =   \left[\begin{array}{cc}
	0 & 2a[\Delta(A_{12})]_{12}\\
	0 & 0
	\end{array}\right],
\end{align}
for all $A_{12} \in \mathfrak{T}_{12}$ and all $A_{11} = \left[\begin{array}{cc}
	a & 0\\
	0 & 0
	\end{array}\right] \in \mathfrak{T}_{11}$. On the other hand, using Lemma \ref{1} (ii) and  Step 2, we have
\begin{align*}
\Delta(A_{11} A_{12}) & = \Delta(A_{11} A_{12} + A_{12} A_{11}) \\ & = 2 \big(\Delta(A_{11})A_{12} + \Delta(A_{12})A_{11}\big) \\ & = 0,
\end{align*}
which means that
\begin{align}
\Delta(A_{11} A_{12})  = 0, \ \ \ \ \ \ \ \ \ \ \ \ \ \ \ \ \ \ \ (A_{11} \in \mathfrak{T}_{11}, \ A_{12} \in \mathfrak{T}_{12}).
\end{align}
Comparing identities (2.13) and (2.14) and using the assumption that $\mathcal{M}$ is 2-torsion free, we deduce that $a [\Delta(A_{12})]_{12} = 0$ for all $a \in \mathcal{A}$ and all $A_{12} \in \mathfrak{T}_{12}$. Using the fact that $m =0$ is the only element in $\mathcal{M}$ satisfying $\mathcal{A} m = \{0\}$, we can infer that $[\Delta(A_{12})]_{12} = 0$ for all $A_{12} \in \mathfrak{T}_{12}$.
\\
\\
\textbf{Step} 6. $\Delta(A) = 0$ for all $A \in \mathfrak{T}$. \\

  It follows from Steps 1-5 that $\Delta(A_{11}) = 0$, $\Delta(A_{12}) = 0$, and $\Delta(A_{22}) = 0$,  for all $A_{11} \in \mathfrak{T}_{11}$, $A_{12} \in \mathfrak{T}_{12}$, and  $A_{22} \in \mathfrak{T}_{22}$. Therefore, for any $A \in \mathfrak{T}$, we  have
\begin{align*}
\Delta(A) = \Delta(A_{11} + A_{12} + A_{22}) = 0.
\end{align*}
Thereby, our proof is complete.
\end{proof}

In the following, we provide an example that shows that the conditions of Theorem \ref{2} are not superfluous.

\begin{example} Let $\mathcal{R}$ be a ring such that the square of each element in $\mathcal{R}$ is zero, but the product of some nonzero elements in $\mathcal{R}$ is nonzero. Let
\begin{align*}
\mathfrak{A} = \Bigg\{\left [\begin{array}{ccc}
0 & a & b\\
0 & 0 & a\\
0 & 0 & 0
\end{array}\right ] \ : \ a, b \in \mathcal{R}\Bigg\}
\end{align*}
Clearly, $\mathfrak{A}$ is a ring with the annihilator
$$
ann(\mathfrak{A}) = \Bigg\{\left [\begin{array}{ccc}
0 & x & y\\
0 & 0 & x\\
0 & 0 & 0
\end{array}\right ] \ : x \in ann(\mathcal{R}), y \in \mathcal{R}\Bigg\},$$ where $ann(\mathcal{R})$ denotes the annihilator of $\mathcal{R}$. We know that the annihilator of any ring is an ideal in that ring. So $ann(\mathfrak{A})$ is an ideal of $\mathfrak{A}$, and we can consider it as a $\mathfrak{A}$-bimodule. Define the mapping $\delta:\mathfrak{A} \rightarrow \mathfrak{A}$ by $$\delta\Bigg(\left [\begin{array}{ccc}
0 & a & b\\
0 & 0 & a\\
0 & 0 & 0
\end{array}\right ]\Bigg) = \left [\begin{array}{ccc}
0 & a & 0\\
0 & 0 & a\\
0 & 0 & 0
\end{array}\right ].$$
It is easy to see that $\delta$ is a commuting Jordan derivation (see \cite[Example 2.8]{H}). Now let
$$
\mathfrak{T} = Tri(\mathfrak{A}, ann(\mathfrak{A}), \mathfrak{A}) = \Bigg\{\left [\begin{array}{ccc}
A & X \\
0 & B
\end{array}\right ] \ : A, B \in \mathfrak{A}, X \in ann(\mathfrak{A})\Bigg\}.$$

Define the mapping $\Delta:\mathfrak{T} \rightarrow \mathfrak{T}$ by $$\Delta\Bigg(\left [\begin{array}{ccc}
A & X\\
0 & B
\end{array}\right ]\Bigg) = \left [\begin{array}{ccc}
\delta(A) & 0\\
0 & 0
\end{array}\right ].$$
A straightforward verification shows that $\Delta(\mathfrak{X}^2) = 2 \Delta(\mathfrak{X})\mathfrak{X} = 2 \mathfrak{X} \Delta(\mathfrak{X})$ for all $\mathfrak{X} \in \mathfrak{T}$. As can be seen, $\Delta$ is a nonzero commuting Jordan derivation of $\mathfrak{T}$, and this demonstrates that the conditions outlined in Theorem \ref{2} are not superfluous.
\end{example}

As mentioned in Section 1, the following result is quite different from that of \cite{P}.

\begin{corollary} Let $\mathcal{A}$ and $\mathcal{B}$ be 2-torsion free rings such that each of them is either
semiprime or satisfies Condition (P) and let $\mathcal{M}$ be as in Theorem \ref{2}. If triangular ring $\mathfrak{T} = Tri(\mathcal{A}, \mathcal{M}, \mathcal{B})$ is commutative, then every Jordan derivation on $\mathfrak{T}$ is identically zero.
\end{corollary}

\begin{corollary} Let $\mathcal{A}$ and $\mathcal{B}$ be 2-torsion free rings such that each of them is either
semiprime or satisfies Condition (P), let $\mathcal{M}$ be as in Theorem \ref{2} and let $\{d_n\}_{n = 0}^{\infty}$ be a commuting Jordan higher derivation on triangular ring $\mathfrak{T} = Tri(\mathcal{A}, \mathcal{M}, \mathcal{B})$, i.e. $d_n(A^2) = \sum_{k = 0}^{n}d_{n - k}(A)d_{k}(A)$, $d_0(A) = A$ and $[d_n(A), A] = 0$ for all nonnegative integers $n$ and any $A \in \mathfrak{T}$. Then $d_n = 0$ for all $n \in \mathbb{N}$.
\end{corollary}

\begin{proof} Using induction on $n$, we prove this corollary. According to Theorem \ref{2}, the result holds trivially for $n = 1$. So it is observed that
\begin{align*}
d_2(A^2) = d_2(A) A + (d_1(A))^2 + A d_2(A) = d_2(A)A + A d_2(A)
\end{align*}
for all $A \in \mathfrak{T}$, which means that $d_2$ is a Jordan derivation on $\mathfrak{T}$. By hypothesis, $[d_2(A), A] = 0$ for all $A \in \mathfrak{T}$ and so $d_2$ is a commuting Jordan derivation on the triangular ring $\mathfrak{T}$. It follows from Theorem \ref{2} that $d_2$ is zero. Let $n$ be an arbitrary positive integer and assume that the result holds for any $k < n$. We prove the result for $n$. In view of our assumption, $[d_n(A), A] = 0$ for all $A \in \mathfrak{T}$. Hence, we have
\begin{align*}
d_n(A^2) = \sum_{k = 0}^{n}d_{n - k}(A)d_{k}(A) = d_n(A)A + A d_n(A) = 2 d_n(A)A = 2 A d_n(A)
\end{align*}
for all $A \in \mathfrak{T}$, which means that $d_n$ is a commuting Jordan derivation on $\mathfrak{T}$. Reusing of Theorem \ref{2} gives the result.
\end{proof}

\begin{corollary} Let $\mathcal{A}$ and $\mathcal{B}$ be 2-torsion free rings such that each of them is either semiprime or satisfies Condition (P), let $\mathcal{M}$ and $\mathfrak{T}$ be as in Theorem \ref{2} and let $B_0 \in \mathfrak{T}$. If $[[A, B_0], A] = 0$ for all $A \in \mathfrak{T}$, then $B_0 \in Z(\mathfrak{T})$, the center of $\mathfrak{T}$.
\end{corollary}

\begin{proof} It is clear that $\Delta_{B_0}: \mathfrak{T} \rightarrow \mathfrak{T}$ defined by $\Delta_{B_0}(A) = [A, B_0]$ is a derivation. Since   $[[A, B_0], A] = 0$ for all $A \in \mathfrak{T}$, $\Delta_{B_0}$ is a commuting derivation on $\mathfrak{T}$. It follows from Theorem \ref{2} that $\Delta_{B_0}$ is zero which yields that $B_0 \in Z(\mathfrak{T})$, as desired.

\end{proof}

Posner proved in \cite[Lemma 3]{P} that if $\mathcal{R}$ is a prime ring and $d$ is a commuting derivation of $\mathcal{R}$, then $\mathcal{R}$ is commutative or $d$ is zero. In the next corollary, we show that every commuting Jordan derivation on a 2-torsion free ring which either is semiprime or satisfies Condition (P) is identically zero under certain conditions.

\begin{corollary} Let $\mathcal{A}$ and $\mathcal{B}$ be 2-torsion free rings such that each of them is either
semiprime or satisfies Condition (P) and let $\mathcal{M}$ and $\mathfrak{T}$ be as in Theorem \ref{2}. Let $d:\mathcal{A} \rightarrow \mathcal{A}$ and $D:\mathcal{B} \rightarrow \mathcal{B}$ be two commuting Jordan derivations and let $\mathfrak{G} : \mathcal{M} \rightarrow \mathcal{M}$ be a mapping which satisfies
\begin{align}
 \mathfrak{G}(am + mb) = 2 d(a)m + 2 \mathfrak{G}(m)b = 2a\mathfrak{G}(m) + 2m D(b)
\end{align}
for all $a \in \mathcal{A}$, $b \in \mathcal{B}$ and $m \in \mathcal{M}$. Then $d$, $D$ and $\mathfrak{G}$ are identically zero.
\end{corollary}

\begin{proof} Define $\Delta: \mathfrak{T} \rightarrow \mathfrak{T}$ by
\begin{align*}
\Delta \left(\left[\begin{array}{cc}
	a & m\\
	0 & b
	\end{array}\right]\right) = \left[\begin{array}{cc}
	d(a) & \mathfrak{G}(m)\\
	0 & D(b)
	\end{array}\right]
\end{align*}
Let $A = \left[\begin{array}{cc}
	a & m\\
	0 & b
	\end{array}\right]$ be an arbitrary element of $\mathfrak{T}$. We have
\begin{align*}
\Delta(A^2) & = \Delta\left(\left[\begin{array}{cc}
	a^2 & am + mb\\
	0 & b^2
	\end{array}\right]\right) = \Delta\left(\left[\begin{array}{cc}
	d(a^2) & \mathfrak{G}(am + mb)\\
	0 & D(b^2)
	\end{array}\right]\right) \\ & = \left[\begin{array}{cc}
	2ad(a) & 2(a \mathfrak{G}(m) + m D(b))\\
	0 & 2b D(b)
	\end{array}\right] \\ & = 2 A \Delta(A).
\end{align*}
One can easily show that $\Delta(A^2) = 2 \Delta(A)A$ for all $A \in \mathfrak{T}$, and this means that $\Delta$ is a commuting Jordan derivation on $\mathfrak{T}$. It follows from Theorem \ref{2} that $\Delta$ is identically zero and this yields that $d$, $D$ and $\mathfrak{G}$ all are   zero, as desired.
\end{proof}

To conclude this paper, we present an example for rings $\mathcal {A}$ and $\mathcal {B}$ and module $\mathcal {M}$ satisfying  the conditions of Theorem \ref{2}.
\begin{example} Let $\mathbb{Z}$ be the set of all integers. Set
\begin{align*}
\mathcal{A} = \mathcal{B} = \left\{\left [\begin{array}{cc}
2n & 0\\
0 & 2n
\end{array}\right ] \ : \ n \in \mathbb{Z} \right\}.
\end{align*}
It is obvious that $\mathcal{A}$ and $\mathcal{B}$ does not contain identity. Let
\begin{align*}
\mathcal{M} = \left\{\left [\begin{array}{cc}
i & j\\
0 & k
\end{array}\right ] \ : \ i, j, k \in \mathbb{Z} \right\}.
\end{align*}
A straightforward verification shows that $\mathcal{A}$ is a semiprime ring which also satisfy Condition (P) and further $\mathcal{M}$ is a faithful $(\mathcal{A}, \mathcal{B})$-bimodule. Moreover $m = 0$ is the only element of $\mathcal{M}$ satisfying $\mathcal{A} m =\{0\}$ (resp. $m \mathcal{B} =\{0\}$). Therefore, the module $\mathcal{M}$ satisfies all the conditions of Theorem \ref{2}.
\end{example}

\section*{Acknowledgement} The authors are greatly indebted to the referee for his/her valuable suggestions and careful reading of the paper.

\vspace{.25cm} {\bf Conflict of interest statement.} \\ The authors state that there is no conflict of interest.

%\vspace{.25cm} {\bf Data availability statement.} \\ The author agrees that the data of this article are available.

\vspace{.25cm} {\bf Author Contribution and Funding Statement.} \\ The authors received no funding for this article.

\bibliographystyle{amsplain}

\end{document}